\documentclass[letterpaper, 10 pt, conference]{ieeeconf}  

\IEEEoverridecommandlockouts 
                                                          
\usepackage{graphics} 
\usepackage{epsfig} 
\usepackage{mathptmx} 
\usepackage{times} 
\usepackage{amsmath} 
\usepackage{amssymb}  
\usepackage{subcaption}
\usepackage[ruled]{algorithm2e}
\usepackage{amsfonts}
\usepackage{multirow}
\usepackage{MnSymbol}
\usepackage{float}
\usepackage{color}

\usepackage{amsbsy}

\usepackage[multiple]{footmisc} 
\newcommand{\x}{\mathbf{x}}
\newcommand{\HH}{\mathbf{H}}

\newtheorem{definition}{Definition}
\newtheorem{theorem}{Theorem}

\newtheorem{proposition}{Proposition}
\newtheorem{assumption}{Assumption}

\title{Conflict-driven Hybrid Observer-based Anomaly Detection}
\author{Zheng Wang$^{1}$, Farshad Harirchi$^{2}$, Dhananjay Anand$^{3}$, CheeYee Tang$^{3}$, James Moyne$^{1}$, Dawn Tilbury$^{1}$ 
\thanks{$^{1}$ Department of Mechanical Engineering, University of Michigan, Ann Arbor, MI, USA \textit{zhengwa, moyne, tilbury@umich.edu}} \thanks{$^{2}$ Department of Electrical Engineering and Computer Science, University of Michigan, Ann Arbor, MI, USA \textit{harirchi@umich.edu}} \thanks{$^{3}$ Software and Systems Divison, National Institute of Standards and Technology, Gaithersburg, MD, USA \textit{dhananjay.anand, cheeyee.tang@nist.gov}}
}

\begin{document}
\maketitle
\thispagestyle{plain}
\pagestyle{plain}

\begin{abstract}
This paper presents an anomaly detection method using a hybrid observer -- which consists of a discrete state observer and a continuous state observer. We focus our attention on anomalies caused by intelligent attacks, which may bypass existing anomaly detection methods because neither the event sequence nor the observed residuals appear to be anomalous. Based on the relation between the continuous and discrete variables, we define three conflict types and give the conditions under which the detection of the anomalies is guaranteed. We call this method conflict-driven anomaly detection. The effectiveness of this method is demonstrated mathematically and illustrated on a Train-Gate (TG) system.
\end{abstract}

\section{Introduction} \label{sec-motivation}
Cyber-Physical Systems (CPS) are systems that are shaped by a combination of computing devices, communication networks, and physical processes \cite{cardenas2009challenges}. The integration of these systems into our every-day life is inevitable. 
The performance and functionality of many critical infrastructures such as power, traffic and health-care networks and smart cities rely on the advances on CPS. A fault or an attack on one of these critical systems, may affect a large portion of society with serious and lethal consequences. As such, the safety and reliability of CPS becomes more and more crucial every day. Fault, attack and anomaly detection mechanisms play a vital role in providing such reliability and safety to CPS. In this paper, we propose an anomaly detection approach that provides formal detection guarantees for an extended class of anomalies in CPS. Similar to \cite{lopez2017categorization}, we refer to any occurrence that is different from what is standard, normal, or expected as \textit{anomaly}. In this paper, we utilize the rich dynamical behavior of mixed continuous and discrete (i.e., hybrid) systems \cite{wan2010composition} as our modeling framework to describe CPS. Even though the design and implementation of anomaly detection methods is significantly more challenging on hybrid models, we leverage these models, because of their advantage in better representing the real-world CPS.    

Our motivational example is a Train-Gate (TG) system, consisting of a train and a gate with a road crossing the track, as shown in Fig.\ref{fig-TGCsys}. It is an abstracted example that captures one of the important characteristics of a railway system which is railway level crossing control system.
The TG system is a hybrid system. The train with an internal controller for the train speed is the continuous system. 
An external controller changes the reference train speed based on the measured train position such that the train passes the gate at a lower speed. 
The gate is a discrete system, which is raised or lowered by a controller using two presence sensors located on both sides of the road. 
If sensor 1 detects the train, the gate must be lowered down to stop traffic on the road. If sensor 2 detects the train, the gate must be raised up to allow traffic on the road. 
Two monitors are used to detect anomalies. One monitor detects anomalies in the continuous train system, which uses the continuous system model and compares the measured variables with the estimated ones. The other monitor detects anomalies in the discrete gate system, which uses the discrete system model and compares the expected discrete event sequence with observed one. If an anomaly is detected from either of these monitors, some actions should be taken to mitigate its impact. 

\begin{figure}[!h]
	\centering
	\includegraphics[width=3in]{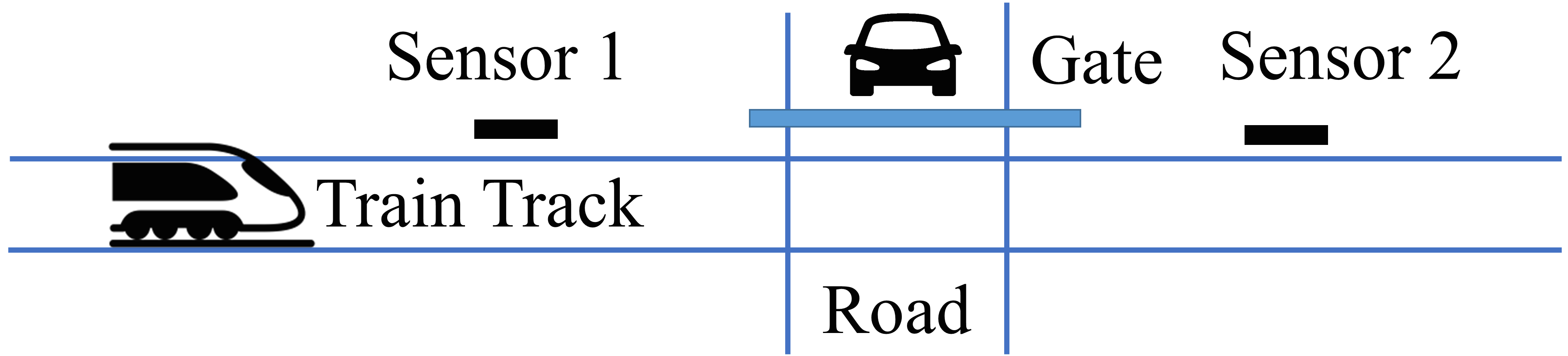}
	\caption{TG Schematic}
	\label{fig-TGCsys}
\end{figure} 

However, an attacker can launch an attack to cause an anomaly bypassing both monitors. For example, an anomalous ramp signal could be added to the measured train position without increasing the difference between the measured and the estimated variables. 
The drifted measured position can make the train pass the gate with a high speed, causing insufficient time to lower the gate. A driver may pass the gate, causing an accident. 

In order to detect this type of anomaly, we propose a higher level monitor to augment the previous two monitors. This new monitor uses a hybrid model of the system, and estimates both the continuous and the discrete variables. For the anomaly of a ramp signal injection on the train position, although the continuous system is anomalous, the discrete system is normal. If sensor 1 detects the train but the estimated train position indicates that the train is far away from sensor 1, a ``conflict'' between the continuous and the discrete variables occurs. 
This new monitor expands the types of anomalies that can be detected by checking the occurrence of conflicts, called conflict-driven method. Both mathematical demonstrations and simulation results illustrate the effectiveness of the conflict-driven method.

\section{Background and Contributions} \label{sec-background}
Various model-based anomaly detection methods have been developed for both continuous systems and discrete systems \cite{ding2008model, sayed2014discrete}. Even though discrete model-based anomaly detection methods are computationally efficient \cite{sampath1996failure}, they cannot provide sufficient resolution of continuous degradations for hybrid systems \cite{goodrich2001continuous}. 
Continuous model-based methods are impractical for the diagnosis of hybrid systems with a large number of discrete states because multiple continuous models need to run in parallel, each model corresponding to one discrete state \cite{zhao2005monitoring}.

As most CPS are hybrid, consisting of both continuous dynamics and discrete behavior, hybrid model-based approaches are promising in anomaly detection. 
Hybrid model-based anomaly detection includes set membership-based methods \cite{harirchi2017active} and observer-based methods \cite{hofbaur2004hybrid}. Given a data trajectory, set membership-based methods check whether the trajectory can be generated by the model. Even though these methods provide necessary and sufficient conditions in some cases for anomaly detection, they are computationally demanding, as they require costly set calculations or mixed integer programming. 
The set membership-based methods are also utilized in active fault diagnosis, where the goal is to design a minimal excitation that guarantees the detection of anomalous behavior \cite{campbell2015auxiliary, nikoukhah2006auxiliary, harirchi2016model}. 

Observer-based methods assume the continuous component of the hybrid model is observable under both normal and anomalous operations. For most observer-based methods, a residual, which is the difference between the estimated output and the actual output, is analyzed to determine the occurrence of an anomaly.
State estimation problem is directly related to observer-based methods. Among various hybrid state estimation methods, a hybrid observer is better for real-time computation since it can reduce the computational complexity \cite{hofbaur2004hybrid}. A hybrid observer consists of two components: a discrete state observer identifying what is the current discrete state, and a continuous state observer estimating the continuous state \cite{hofbaur2004hybrid,balluchi2002design}. With the hybrid observer framework, various traditional residual-based anomaly detection methods can be applied for hybrid systems, including different residual generation methods, such as the dedicated and generalized observer scheme \cite{clark1978instrument, wang2017improved}, and some residual evaluation methods, such as adaptive threshold \cite{emami1988effect}.

\subsection{Contributions}
Even though the residual-based methods are efficient, intuitive and easy to implement, they can easily be tricked by a smart attacker or by sensor faults that make the continuous system unobservable, causing anomalies. An example of such class of anomalies is described in Section \ref{sec-motivation}. 
In this paper, we propose a conflict-driven anomaly detection approach with three conflict types defined based on the relation between the discrete and the continuous variables of the hybrid systems
and in addition to faults that can be detected by traditional observer-based and residual-based methods, it is capable of providing guarantees on the detection of attacks and faults that are undetectable using these methods.

\section{Problem Formulation} \label{sec-problem}
In this section, we describe the modeling framework that we consider and the anomaly types that are of interest. Also, a review of utilized hybrid observer is given.

\subsection{Notation} \label{sec-notation}
Let $\|\cdot\|$ denote $\infty$-norm, $\tilde{\cdot}$ denote estimated variables, $\cupdot$ denote disjoint union, and $\Box \sigma$ denote the ball of center 0 and of radius $\sigma$. In addition, $\mathbf{x}\in \mathbb{R}^n$ represents a vector, where its $i^{th}$ element is indicated by $\mathbf{x}^{(i)}$. $\mathbf{A}
\in \mathbb{R}^{n\times m}$ represents a matrix. The linear span of a set of vectors is denoted by $span(\cdot)$. For a set $X \subset \mathbb{R}^n$, we denote its closure, interior, and boundary by $\overline{X}$, $X^o$ and $\partial X$ respectively. Clearly, $\partial X=\overline{X}\backslash X^o$. The volume of the closed set $\overline{X}$ is denoted by $Vol(\overline{X})$.
\subsection{Modeling Framework} \label{sec-systems}
\subsubsection{Hybrid Model} A hybrid system can be modeled as a hybrid automaton $\mathcal{H} = (\mathcal{X},\mathcal{U},\mathcal{Y}, Init, field, E, \phi, \eta)$, where each element is defined as
\begin{itemize}
\item $\mathcal{X} = Q \times X$: a set of discrete and continuous states
\item $\mathcal{U}=\Psi \times U$: a set of discrete and continuous inputs
\item $\mathcal{Y} = \Omega \times Y$: a set of discrete and continuous outputs
\item $Init = (q(t_0), \mathbf{x}(t_0)) \in \mathcal{X}$: an initial state
\item $field: \mathcal{X} \times \mathcal{U} \rightarrow X$: a time invariant vector field
\item $E=\Psi\cupdot \Omega$: a set of discrete events
\item $\phi: Q \times \Psi \rightarrow Q$: a set of discrete transitions
\item $\eta: \mathcal{X}\times \mathcal{U} \rightarrow \mathcal{Y}$: an output map consisting of a discrete output map $\zeta$ and a continuous output equation $h$

$\zeta: Q \times \Psi \rightarrow \Omega$: a discrete output map

$h: \mathbf{y}(t) = \mathbf{C}_q\mathbf{x}(t)+\mathbf{v}(t)$: a continuous output equation
\end{itemize}

The hybrid models considered in this paper capture both nominal system model with a set of nominal discrete states $Q_n$ and anomaly models with a set of anomalous discrete states $Q_f$. The set of all discrete states is
defined as $Q = Q_n \cupdot Q_f$. The nominal hybrid system $\mathcal{H}_n$ can be derived by removing $Q_f$ and the events and transitions connecting $Q_f$. 
The initial state $Init$, which is a combination of initial discrete state $q(t_0) \in Q_n$ and initial continuous state $\mathbf{x}(t_0)$, is not required to be known. 

For each discrete state $q\in Q$, we consider continuous dynamics that can be represented by a Linear Time Invariant (LTI) model, subject to process and measurement noise.
\begin{equation} \label{eq-linearsys}
	\begin{aligned}
	field:& \mathbf{x}(t+1)=\mathbf{A}_q\mathbf{x}(t)+\mathbf{B}_q\mathbf{u}(t)+\mathbf{w}(t),\\
    h: &\mathbf{y}(t)=\mathbf{C}_q \mathbf{x}(t)+\mathbf{v}(t),
	\end{aligned}
\end{equation}
where $\mathbf{A}_q\in \mathbb{R}^{n\times n}, \mathbf{B}_q\in \mathbb{R}^{n\times n_u}, \mathbf{C}_q \in \mathbb{R}^{n_y\times n}$ are system matrices, $\mathbf{x}\in X\subset \mathbb{R}^{n}$, $\mathbf{u}\in U\subset \mathbb{R}^{n_u}$ and $\mathbf{y}\in Y\subseteq \mathbb{R}^{n_y}$ are continuous states, inputs and outputs, respectively. The process and measurement noise are represented by $\mathbf{w}\sim \mathcal{N}(0,\mathbf{W})$ and $\mathbf{v}\sim \mathcal{N}(0,\mathbf{V})$, respectively, where $\|\mathbf{w}\|\leq w$ and $\|\mathbf{v}\|\leq v$. Each entry of the process and measurement noise has its bound, i.e., $|\mathbf{w}^{(i)}|\leq w_i$ and $|\mathbf{v}^{(i)}|\leq v_i$. The continuous dynamical models of the system in anomalous discrete states are not required to be known. 
To simplify the notation, we assume:
\begin{assumption} \label{assump-cmatrix}
The output matrix $\mathbf{C}_q$ is an identity matrix in all discrete states, i.e., $\forall q\in Q, \mathbf{C}_q=\mathbf{I}$.
\end{assumption}
We can easily extend our work to general $\mathbf{C}$ matrix assuming the continuous system is observable.

Discrete events $E$ can be partitioned into observable events $E_o$ and unobservable events $E_{uo}$, i.e., $E=E_o\cupdot E_{uo}$. Only observable events can be detected by an observer. We denote the set of observable input events as $\Psi_o$ and a set of unobservable input events as $\Psi_{uo}$. Obviously, all of the output events are observable.  

The $i^{th}$ discrete event occurs at time $t_i$. The continuous evolutions occur in time $t\in [t_{i-1}+1, t_i],\forall i=1,2,...$.
In reality, discrete events may occur between two adjacent sample times. We assume 
\begin{assumption} \label{assump-inputevent}
The occurrence of the discrete events can be captured at sample times. At most one input event occurs within one sampling period. An output event occurs simultaneously with an input event.  
\end{assumption}
Note that the discrete state is changed one time step after a discrete input event occurs, that is $\phi(q(t_i),\psi)=q'(t_i+1)$, where $q(t_i), q'(t_i+1)\in Q$. To each discrete transition $\phi(q,\psi)=q'$, we associate a guard:
\begin{equation}
G(q,q',\psi)=\{\mathbf{x}: s_G\mathbf{x}^{(i_G)}\geq c_G\},
\end{equation}
where $c_G$ is a constant value and $s_G$ is either $-1$ or $1$. A guard is a closed half-space divided by the hyperplane 
\begin{equation} \label{eq-hyperplane}
\begin{aligned}
\mathcal{P}(q,q',\psi)=\{\mathbf{x}: \mathbf{x}^{(i_G)}= s_Gc_G\}.
\end{aligned}
\end{equation}
A guard $G(q,q',\psi)$ indicates that the transition $\psi$ will take place if and only if the $i_G^{th}$ state variable of $s_G\mathbf{x}$ is no smaller than $c_G$ in discrete state $q$. 

To each discrete state $q \in Q$, we associate an invariant:
\begin{equation} \label{eq-invariant}
Inv_{q} = \{\mathbf{x}: \forall i=1,...,n, \underline{\beta}_i\leq \mathbf{x}^{(i)}\leq \overline{\beta}_i, \}\subseteq X,
\end{equation} 
where $\underline{\beta}_i$ and $\overline{\beta}_i$ are constant values. An invariant is a hyperrectangle with bounded intervals on each continuous state variable. An invariant $Inv_{q}$ indicates that the system can remain in the discrete state $q$ if and only if the continuous state $\mathbf{x} \in Inv_{q} \backslash \bigcup_j G(q,q_j,\psi_j)$. 

Our definitions of guard $G(q,q',\psi)$ and invariant $Inv_q$ indicate that $c_G$ is between the lower and upper bounds of the state variable $\mathbf{x}^{(i_G)}$ of the invariant $Inv_q$, i.e., $\underline{\beta}_{i_G}\leq c_G \leq \overline{\beta}_{i_G}$. We define a neighbor hyperplane of guard $G(q,q',\psi)$ as
\begin{definition}
(Neighbor hyperplane of guard $G(q,q',\psi)$) is one of the hyperplanes forming the boundary of the invariant $\partial Inv_q$, which is defined as follows: 
\begin{equation} \label{eq-neighbor}
\begin{aligned}
\mathcal{L}(q,q',\psi) = \{\mathbf{x}\in X: & |\mathbf{x}^{(i_G)}-\mathbf{x}'^{(i_G)}|=\min(c_G-\underline{\beta}_{i_G}, \overline{\beta}_{i_G}-c_G) \\
&\land \mathbf{x}^{(i_G)} \in \{\underline{\beta}_{i_G},\overline{\beta}_{i_G}\}, \mathbf{x}'\in \mathcal{P}(q,q',\psi)\}.
\end{aligned}
\end{equation}
\end{definition}
An example of neighbor hyperplane of $G(q,q',\psi)$ is shown in Fig.\ref{fig-getregions}.
To simplify notation, we denote $c_{\mathcal{L}}$ as the value of $\mathbf{x}^{(i_G)}$, where $\mathbf{x}\in \mathcal{L}(q,q',\psi)$.
If $\mathcal{P}(q,q',\psi)$ forms one of the hyperplanes of $\partial Inv_q$, then $\mathcal{L}(q,q',\psi)= \mathcal{P}(q,q',\psi)$. Otherwise, $\mathcal{L}(q,q',\psi)\cap \mathcal{P}(q,q',\psi)=\emptyset$. Discontinuities may exist in continuous variables due to discrete transitions in general hybrid systems. However, in our hybrid system formalism, no discontinuities exist in the continuous variables. This is imposed without any reset maps. 

The hybrid observer used in this paper is proposed in \cite{balluchi2002design}, which is designed based on the Finite State Machine (FSM) associated with the nominal hybrid model. The FSM $\mathcal{M}_n$ is derived by removing all of the continuous dynamics in $\mathcal{H}_n$, and is represented by tuple $(Q, \Psi, \Omega, q(t_0), E, \phi, \zeta)$. 
In order to get a unique estimate of the discrete state with the hybrid observer after finite observable events, we assume
\begin{assumption} \label{assump-cso}
The FSM $\mathcal{M}_n$ is current-state observable.
\end{assumption}
Current-state observable is defined in \cite{balluchi2002design}. 
\begin{definition}
(Current-State Observable) A FSM is current-state observable if there exists an integer $k$ such that for any unknown initial discrete state, the discrete state at $i$ can be determined from the observed input/output event pairs sequence up to $i$, i.e., $i\geq k$. 
\end{definition}
Note that one input/output event pair is considered as one input event to the hybrid observer. 
Thus, after the $k^{th}$ input/output event pair occurs, the hybrid observer can give a unique estimated discrete state.
The necessary and sufficient condition of current state observability is given in \cite{balluchi2002design}.

\subsubsection{Nominal Discrete States} \label{sec-nominalstate}
We partition the invariants of the nominal discrete states into an intermediate region and several normal operating regions. The intermediate region $\mathcal{R}_{in}$ is the union of all the intersections between the invariants of any two nominal discrete states
\begin{equation}
\begin{aligned}
\mathcal{R}_{in} = \{\mathbf{x}\in X: \forall q,q'\in Q_n, q\neq q'\land \mathbf{x}\in Inv_q\cap Inv_{q'} \}.
\end{aligned}
\end{equation}

For discrete state $q\in Q_n$, we define a normal operating region as the set of continuous states that are in the invariant but not the intermediate region, 
\begin{equation}
\begin{aligned}
\mathcal{R}_{no,q} = \overline{Inv_q}\backslash \overline{\mathcal{R}_{in}}.
\end{aligned}
\end{equation}
 
We pose an assumption on the hybrid model to restrict state-space abstraction method. This assumption helps select the appropriate hybrid model of the system with which the conflict-driven method can provide detection guarantees. 
\begin{assumption} \label{assump-nonintersecting}
The intermediate region is bounded by the hyperplane $\mathcal{P}(q,q',\psi)$ corresponding to the guard $G(q,q',\psi)$ and the neighbor hyperplane of $G(q,q',\psi)$ and $\partial Inv_q$ in each discrete state.
\begin{equation}
\begin{aligned}
\mathcal{R}_{in} \subset \bigcup_{q\in Q_n} \{\mathbf{x}\in Inv_q: & \forall q_j\in Q_n: \exists \mathcal{P}(q,q_j,\psi_j), \\
& \min(c_G, c_\mathcal{L})\leq \mathbf{x}^{(i_G)}\leq \max(c_G, c_\mathcal{L})\}.
\end{aligned}
\end{equation}
\end{assumption}
The visualization of this assumption on a $2$-dimensional system is shown in Fig.\ref{fig-getregions}. Assumption \ref{assump-nonintersecting} indicates $\mathcal{L}(q,q',\psi)$ is one of the hyperplanes forming $\partial Inv_{q'}$.

\begin{figure}[thpb]
      \centering
      \includegraphics[width=1.8in]{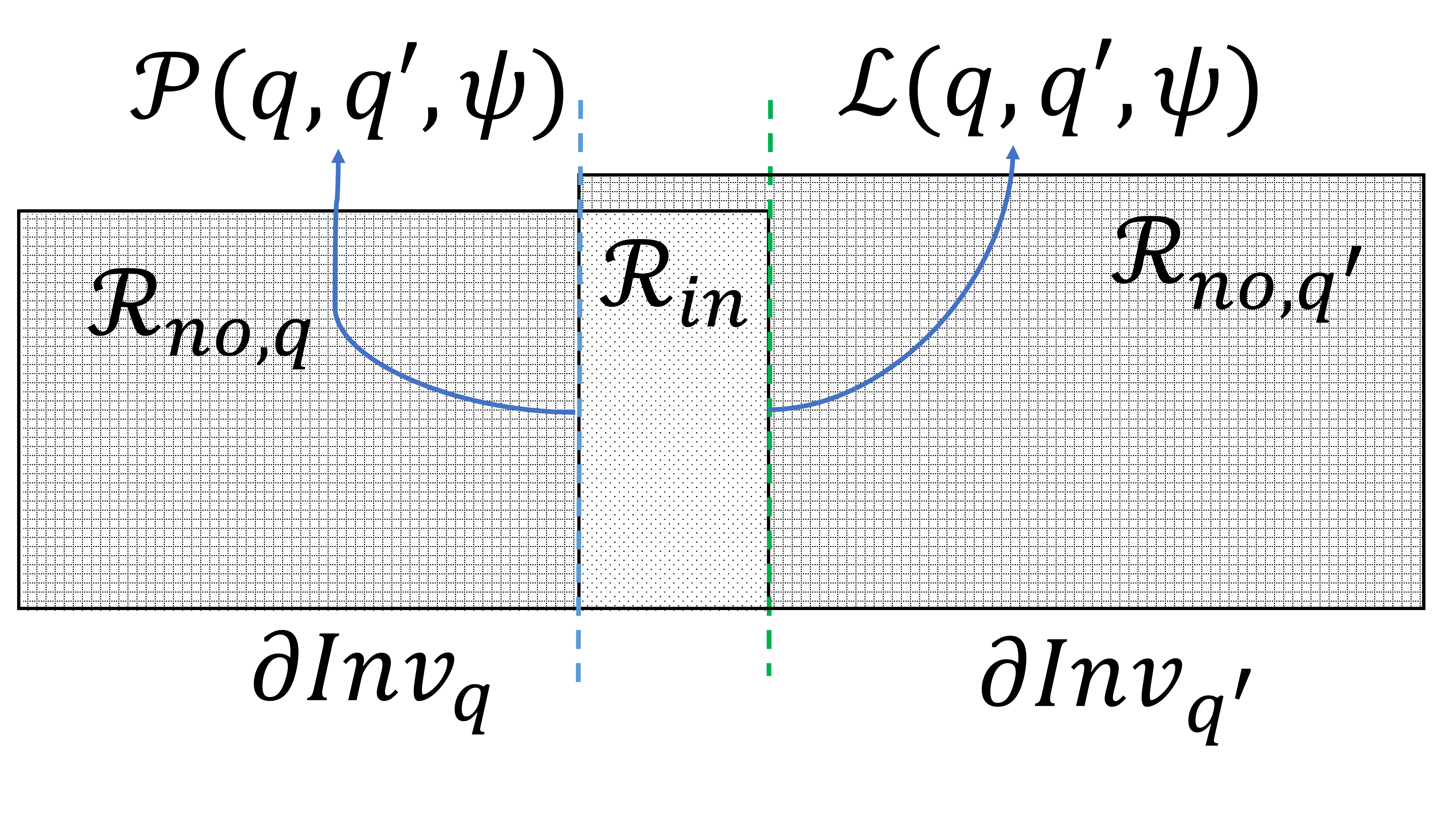}
      \caption{Normal operating and intermediate regions} 
      \label{fig-getregions}
\end{figure}

The basic principle of the conflict-driven method is to check at each time step, whether or not the sets of continuous states, which are calculated based on the estimated continuous state, intersect with the invariant of the estimated discrete state. The sets of continuous states include an initial set with the estimated continuous state as the center, and a forward reachable set which is the set of all continuous states that can be reached along trajectories starting in the initial set. 
Reachable set calculation requires following assumptions. 
\begin{assumption} \label{assump-stable}
The continuous system in each nominal discrete state is open-loop stable or marginally stable, i.e., $|\lambda(\mathbf{A}_{q})|\leq 1$, where $\lambda(\mathbf{A}_{q})$ are the eigenvalues of $\mathbf{A}_{q}$. 
\end{assumption}
\begin{assumption} \label{assump-input}
The continuous input signal is bounded, and the bound is known, i.e., $||\mathbf{u}||\leq \mu$. 
\end{assumption}

A great deal of attention has been given to algorithms and software developed for analysis of hybrid systems. To date, the most efficient way to compute the reachable set is to use zonotopes \cite{girard2005reachability}. A zonotope is a Minkowski sum of a finite set of line segments, defined as
\begin{definition}
(Zonotope $Z$) is a set such that:
\begin{equation}
\begin{aligned}
Z =& (\mathbf{x_c}, <\mathbf{g_1},..., \mathbf{g_p}>) \\
=& \{\mathbf{x}\in \mathbb{R}^n:\mathbf{x} = \mathbf{x_c}+\Sigma_{i=1}^{i=p}b_i\mathbf{g_i}, -1\leq b_i \leq 1 \}, p\geq n,
\end{aligned}
\end{equation}
where $\mathbf{x_c}, \mathbf{g_i} \in \mathbb{R}^n$ are the center and generators, respectively. 
\end{definition}
Both $p$ and $n$ determine the maximum number of vertices and facets. 

\subsection{Hybrid Observer} \label{sec-hybridobs}
Given the nominal hybrid model $\mathcal{H}_n$, we can design a hybrid observer to estimate both the discrete state and the continuous state of the system using the method in \cite{balluchi2002design}.
The hybrid observer $\mathcal{O}$ consists of a discrete state observer $\mathcal{D}$ and a continuous state observer $\mathcal{C}$, as shown in Fig.\ref{fig-relation}. The discrete state observer receives discrete input/output event pair $(\psi,\omega)$ and gives $\tilde{q}$. 
The estimated discrete state $\tilde{q}$ contains a set of estimated discrete states before the occurrence of the $k^{th}$ observable input/output event pair.
After the occurrence of the $k^{th}$ observable input/output event pair, $\tilde{q}$, which contains a unique estimate, is passed to the corresponding continuous state observer. Then the continuous state observer gives an estimated continuous state $\tilde{\mathbf{x}}$ using the continuous input $\mathbf{u}$ and output $\mathbf{y}$.

The discrete state observer is represented by a FSM which is a tuple $\mathcal{D} = (\tilde{Q}, E_{\mathcal{D}}, -, Q_n, E_{\mathcal{D}}, \tilde{\phi}, -)$, where $E_{\mathcal{D}} = (\Psi, \Omega)$ is the set of discrete input/output event pairs of $\mathcal{M}_n$. The discrete state observer is tracking the set of possible discrete states that the system can be in. Therefore, no discrete output events or discrete map are defined for discrete state observer.

The construction of $\mathcal{D}$ starts from $\tilde{q}(t_0)$: with unknown initial discrete state of $\mathcal{M}_n$, $\tilde{q}(t_0)=Q_n$. Then for each discrete state $\tilde{q} \in \tilde{Q}$, we identify the input/output event pairs $(\psi,\omega)$, that label all the transitions out of any state $q'$ in $\tilde{q}$. These events are called active event set of $\tilde{q}$. For each pair $(\psi,\omega)$ in the active event set, we identify $q\in Q_n$ that can be reached from $q' \in \tilde{q}$, and these states return as a new $\tilde{q}$ in $\tilde{Q}$. This transition is added to $\tilde{\phi}$ satisfying
\begin{equation}
\begin{aligned}
\tilde{\phi}&:= \{q\in Q_n:\exists q'\in \tilde{q}, \, q\in \phi(q',\psi)\land \omega = \zeta(q',\psi)\}.
\end{aligned}
\end{equation}
Repeat this step until no new $\tilde{q}$ and $\tilde{\phi}$ can be added to $\mathcal{D}$.

e state observer. In this step, we identify the input and output event pair $(\psi,\omega)$, that label all the transitions out of any state $q'$ in $\tilde{q}$. These events are called active event set of $\tilde{q}$. For each pair $(\psi,\omega)$ in the active event set, we identify $q\in Q_n$ that can be reached from $q' \in \tilde{q}$, and these states return as a new $\tilde{q}$ in $\tilde{Q}$. This transition is added to the transition function of observer $\mathcal{D}$. The output function $\tilde{\zeta}$ equals to $\tilde{\phi}$ because the output of the discrete state observer is the estimated discrete state. Note that $\tilde{Q}$ may contain a set of discrete states of the plant system.

To reduce the effect of system noise on state estimation, we use a Kalman filter as the continuous state observer, with Kalman gain $\mathbf{K}_{q(t)}$. It is well known that the Kalman gain will converge in a few steps in practice if the system is observable \cite{mo2010false}. 
We can use the steady state Kalman gain given in \cite{mo2010false}, with which the eigenvalues of $(\mathbf{A}_q-\mathbf{K}_q\mathbf{A}_q)$ are stable. Note that we have different Kalman gains for different discrete states. Let us define
\begin{definition}
(Dwell time $\Delta t$) is the minimum time to guarantee the convergence of the estimation error. 
\end{definition}
Dwell time $\Delta t$ should satisfy the condition in section 3.2 in \cite{balluchi2002design}, Then we assume: 
\begin{assumption}
The time gap between any two consecutive transitions is greater than dwell time, i.e., $t_{i}-(t_{i-1}+1)>\Delta t$.
\end{assumption}
With bounded noise, we design Kalman filter such that the estimation error $\mathbf{x_e}(t)=\mathbf{x}(t)-\tilde{\mathbf{x}}(t)$ is bounded when the Kalman filter reaches its steady state, i.e., $\exists t_{ss}, \|\mathbf{x_e}(t)\| \leq \theta, t>t_{ss}$. The residual $\mathbf{r}$ of the system is defined as the difference between the measure output and the estimated output,
\begin{equation} \label{eq-residual}
\begin{aligned}
\mathbf{r}(t)=&\mathbf{y}(t)-\tilde{\mathbf{x}}(t).
\end{aligned}
\end{equation}
In the nominal discrete states, the residual $\mathbf{r}(t), t>t_{ss}$ is bounded by $\theta+v$ because of bounded estimation error and noise. If $\|\mathbf{r}(t)\|>\theta+v, t>t_{ss}$, then the system is in an anomalous discrete state.

\subsection{Anomalous Discrete States}\label{sec-anomalousstate}
An anomaly $f\in \Psi_{uo}$ is defined as an unobservable input event that transits the system from a nominal discrete state $q_n\in Q_n$ to an anomalous discrete state $q_f\in Q_f$. Arguably, the multiplicative anomalies can be represented by additive anomaly models (e.g., Section 3.5 in \cite{ding2008model}). Thus, we restrict our attention to additive anomaly models as follows.
\begin{equation} \label{eq-additiveanomaly} 
\begin{aligned}
\mathbf{y}(t)&=\mathbf{x}(t)+\mathbf{v}(t) +\pmb{\Gamma }\pmb{\gamma}(t),
\end{aligned}
\end{equation}
where $\pmb{\Gamma}\in \mathbb{R}^{n\times n}$ is a diagonal matrix with binary variables. The $i^{th}$ diagonal variable is 1 if and only if the $i^{th}$ output is added with an anomalous signal $\pmb{\gamma}(t)\in \mathbb{R}^n$. 
Then the residual in anomalous discrete states is changed to 
\begin{equation} \label{eq-residualanomalous}
\begin{aligned}
\mathbf{r}(t)=\mathbf{x_e}(t)+\mathbf{v}(t)+\pmb{\Gamma\gamma}(t).
\end{aligned}
\end{equation}
The conflict-driven method is guaranteed to detect the anomalies that are not consistent with the continuous dynamics of the system, i.e, the anomalies that make the residual greater than threshold $\theta+v$. This is because of leveraging continuous state observer that is described in the previous subsection. Additionally, the proposed method extends the types of anomalies that can be detected compared to the methods mentioned in Section \ref{sec-background}. 

Perfectly attackable systems are defined by Mo, et al. in \cite{mo2010false} as continuous systems for which anomalies caused by certain attacks can remain undetected, i.e., the residual will not increase. One of the conditions for a continuous system to be perfectly attackable is that the state matrix $\mathbf{A}_{q_f}$ has at least one unstable or marginally stable eigenvalue. 
If the continuous system only has stable eigenvalues, anomalies on the system will increase the residual.
The smart attacks that cannot be detected in perfectly attackable systems are called False Data Injection Attack (FDIA) as defined and demonstrated in \cite{mo2010false}. One of the conditions of FDIA is that the eigenvector $\pmb{\xi}$ corresponding to an unstable or marginally stable eigenvalue of $\mathbf{A}_{q_f}$ is in the span of $\pmb{\Gamma}$, i.e., $\pmb{\xi}\in span(\pmb{\Gamma})$. If $\pmb{\xi}\not \in span(\pmb{\Gamma})$, the anomaly will increase the residual and will be detected by the Kalman filter implemented as the continuous state observer in conflict-driven method. 

As mentioned before, detecting FDIA type anomalies is challenging, as their effect cannot be observed in the value of residual. In addition to anomalies that can be detected by checking the residual, our main contribution is to also guarantee the detection of this particular type of anomalies, if they satisfy certain conditions (explained in Section \ref{sec-solution}). Let us define Type-$C_u$ anomalies for the hybrid systems as:

\begin{definition}
(Type-$C_u$ anomaly) is an anomaly that is caused by False Data Injection Attack. If an anomaly occurs at time $t_f$, it satisfies the following two conditions.

1) The input-output sequence generated from the anomalous discrete state satisfies the continuous dynamics of the nominal discrete states for $t \geq t_f$, that is, the residual does not grow larger than the threshold $\theta+v$.

2) The occurrence of the anomaly results in:
	\[
	\text{for  } t\geq t_f, \text{ if } q\in Q_f \implies \|\mathbf{x_e}(t)\|>\theta .
	\]
\end{definition}

Our objective is to extend the detection guarantees to the class of Type-$C_u$ anomalies. In order to establish the goal, we also assume that:
\begin{assumption}\label{assump-anomaytime}
An anomaly occurs after the continuous state observer enters its steady state, i.e., $t_f\geq t_{ss}$.
\end{assumption}

\section{Conflict-driven Anomaly Detection Method} \label{sec-solution}
In the conflict-driven method, we define three conflict types. This method checks the occurrence of the conflicts to detect anomalies. The work flow diagram is shown in Fig.\ref{fig-relation}. Note that this method is used after the hybrid observer is in the steady state, i.e., $t\geq t_{ss}$.

\begin{figure}[thpb]
      \centering
      \includegraphics[width=3.25in]{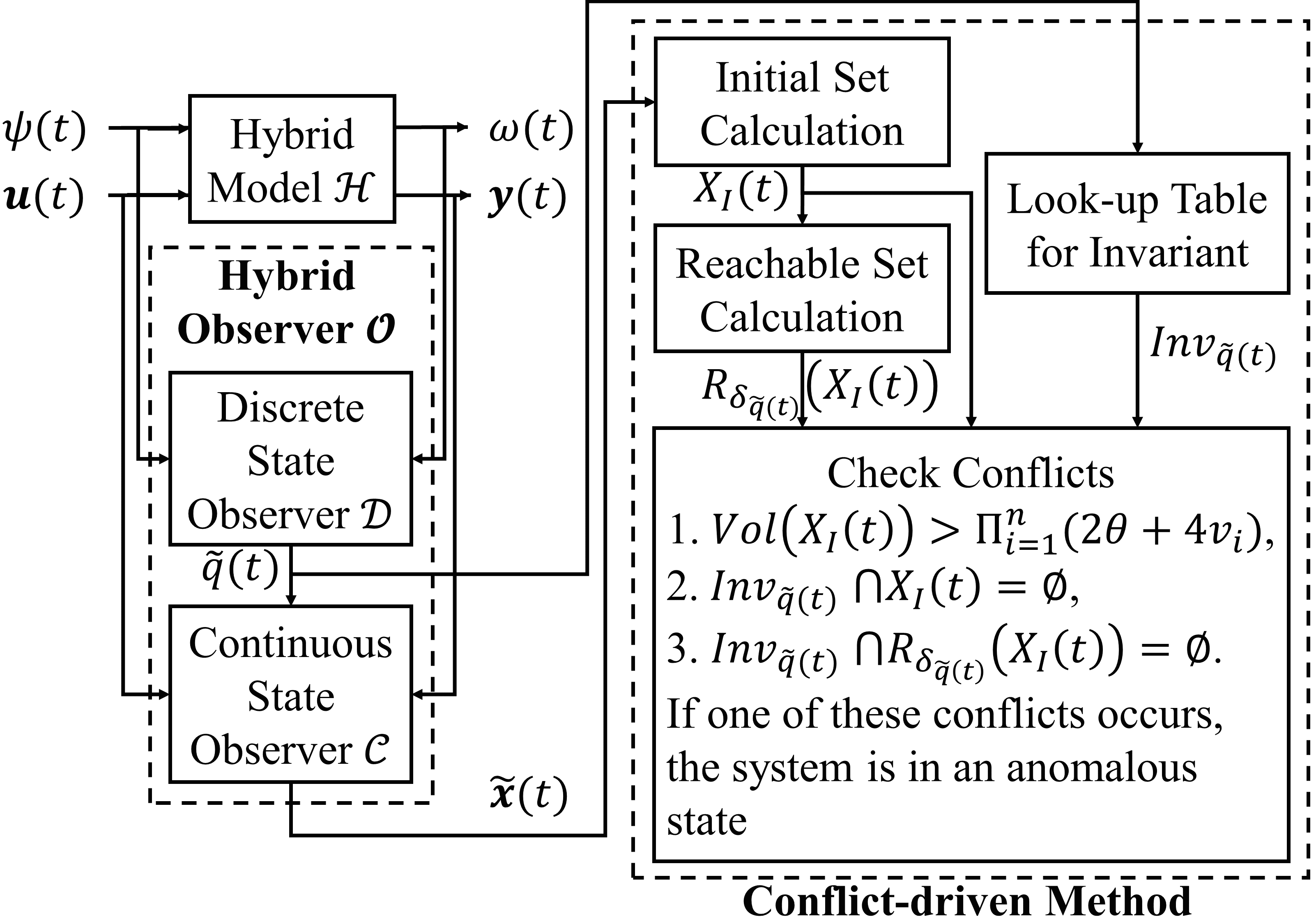}
      \caption{Conflict-driven anomaly detection method}
      \label{fig-relation}
\end{figure}

The conflict-driven method has three steps:

\subsubsection{Calculate an initial set $X_I(t)$}

An initial set $X_I(t)$ is constructed as a zonotope based on $\tilde{\mathbf{x}}(t)$ and $\mathbf{r}(t)$, as $X_I(t) = (\tilde{\mathbf{x}}(t), <\mathbf{g_1}, ..., \mathbf{g_n}>)$. The $i^{th}$ generator $\mathbf{g_i}^{(i)} = |\mathbf{r}^{(i)}(t)|+v_i$. Other entries of vector $\mathbf{g_i}$ are zero. Based on Equation \eqref{eq-residual}, we have $|\mathbf{x_e}^{(i)}(t)|\leq |\mathbf{r}^{(i)}(t)|+v_i$ in nominal discrete states. Thus, we can ensure $\mathbf{x}(t)\in X_I(t)$ when the system is in nominal discrete states. The initial set is changing at each time step because of the changes in the estimated continuous state and the residual.  

\subsubsection{Calculate the reachable set $R_{\delta_{\tilde{q}(t)}}(X_I(t))$}

The $\delta_{\tilde{q}(t)}$ time-step forward reachable set $R_{\delta_{\tilde{q}(t)}}(X_I(t))$ is calculated starting from $X_I(t)$ constructed in Step 1. It satisfies
\begin{equation}
\begin{aligned}
R_{\delta_{\tilde{q}(t)}}(X_I(t)) \subseteq \mathbf{A}_{\tilde{q}(t)}^{\delta_{\tilde{q}(t)}} X_I(t)+\Box \sigma_{\tilde{q}(t)}
\end{aligned}
\end{equation}
where $\sigma_{\tilde{q}(t)}=\frac{1-\|\mathbf{A}_{\tilde{q}(t)}\|^{\delta_{\tilde{q}(t)}}}{1-\|\mathbf{A}_{\tilde{q}(t)}\|}(\|\mathbf{B}_{\tilde{q}(t)}\|\mu+w)$. For more details about reachable set calculation using zonotopes, refer to \cite{girard2005reachability}.

\subsubsection{Check conflicts}

We define three conflict types in this paper, as shown in Fig.\ref{fig-conflicts}:

Conflict $A$. The volume of the initial set is larger than the bound, i.e., $Vol(X_I(t))>\Pi_{i=1}^n(2\theta+4v_i)$

Conflict $B$. The initial set has no intersection with the invariant of the estimated discrete state ($X_I(t)\cap Inv_{\tilde{q}(t)} = \emptyset$)

Conflict $C$. The $\delta_{\tilde{q}(t)}$ time steps forward reachable set has no intersection with the invariant of the estimated discrete state, i.e., $R_{\delta_{\tilde{q}(t)}}(X_I(t)) \cap Inv_{\tilde{q}(t)} = \emptyset$.

If one of these conflicts occurs, the system is in an anomalous state. 

\begin{figure}[thpb]
      \centering
      \includegraphics[width=3.2in]{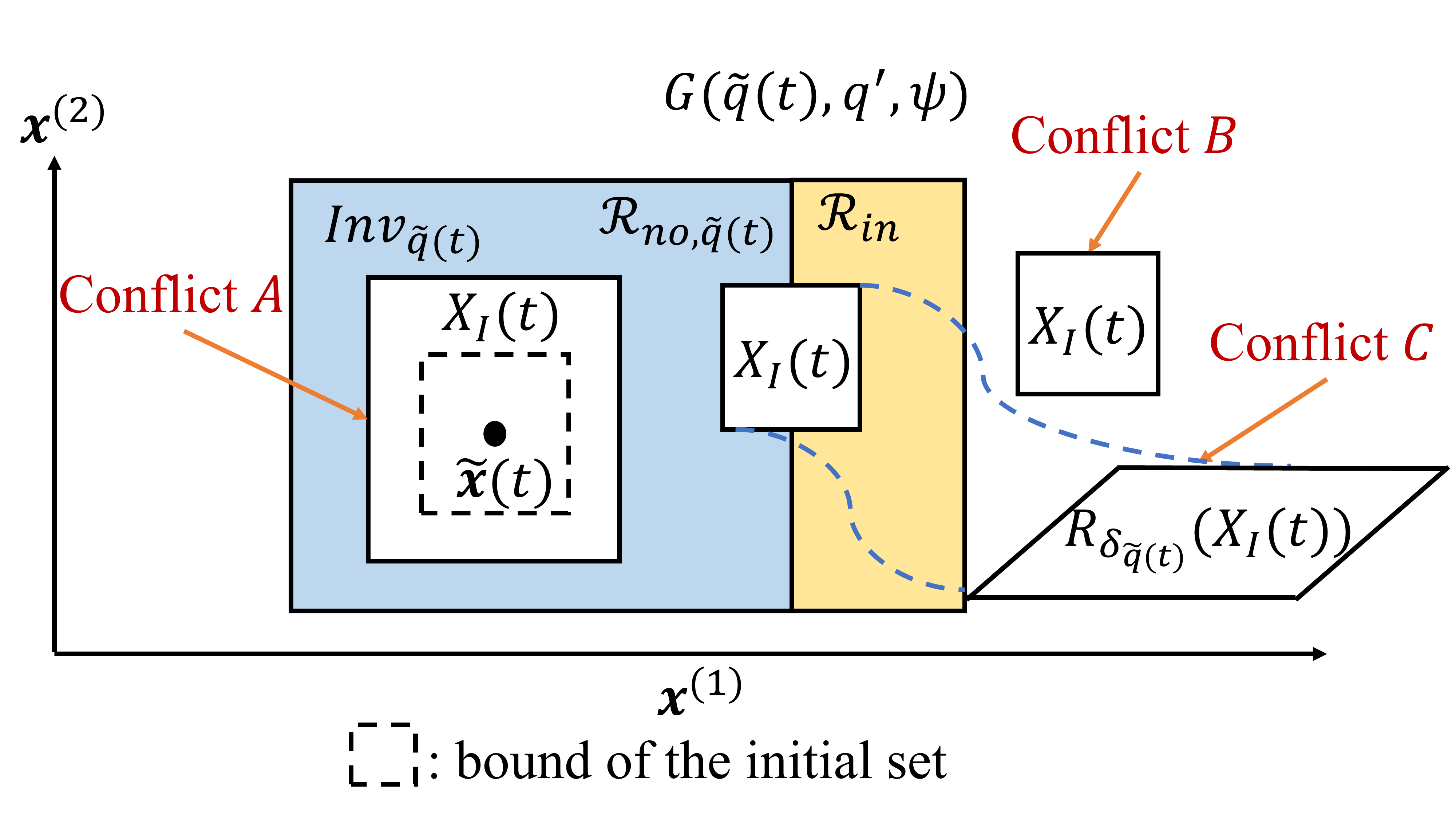}
      \caption{Three conflict types}
      \label{fig-conflicts}
\end{figure}

Note that for Step 2, we do not consider the discrete behavior in reachability analysis. The reachable set could be completely outside the invariant if $\delta_{\tilde{q}(t)}$ is too large, causing false alarms. To avoid false alarms and provide detection guarantees, we determine the time steps $\delta_{q}$ for each discrete state with given nominal hybrid model $\mathcal{H}_n$ of the system according to the following steps:

1) In $Inv_q$, starting from the intersection of the hyperplane corresponding to the $i^{th}$ guard $G(q,q_i,\psi_i)$ as defined by \eqref{eq-hyperplane} and $Inv_q$, we find the minimum time steps $\delta_{q,i}$ which satisfies
\begin{equation} \label{eq-delta1}
\begin{aligned}
R_{\delta_{q,i}+1}(\mathcal{P}(q,q_i,\psi_i)\cap Inv_q)  &\cap \mathcal{L}(q,q_i,\psi_i) \neq \emptyset 
\end{aligned}
\end{equation}
Note that $\delta_{q,i}$ may be different for different guards in the same discrete state. The reason we use $\delta_{q,i}+1$ is that the continuous system is a discrete-time model and we want to ensure the $\delta_{q,i}$ time-step forward reachable set, starting from any possible real continuous state when a transition occurs, has intersection with $Inv_q$ in nominal discrete states.

2) Let $\delta_{q}=\min_i(\delta_{q,i})$. If the distance between $\mathcal{P}(q,q_i,\psi_i)$ and $\mathcal{L}(q,q_i,\psi_i)$ is small, $\delta_{q}$ may be $0$. Then we only need to check Conflicts $A$ and $B$ in discrete state $q$.

Following proposition and theorem demonstrate the effectiveness of the conflict-driven method. We give some intuitions first. Proposition \ref{lemma-initial} gives the upper bound for the volume of the initial set. Based on Assumption \ref{assump-cmatrix}, in a nominal discrete state, the estimation error of the continuous state, as well as the residual, should converge. Therefore, an upper bound exists for the volume of the initial set $Vol(X_I(t)), t>t_{ss}$, as demonstrated in Proposition \ref{lemma-initial}. The increase of $Vol(X_I(t))$ indicates the increase of the residual. Conflict $A$ can detect anomalies that increase the residual. Since the main contribution of this paper is focusing on the detection of Type-$C_u$ anomaly which does not increase the residual, finding the lower bound of the anomalous signal which causes conflict $A$ and the conditions under which a residual-based method is equivalent to checking Conflict $A$ are part of our future work. 

\begin{proposition} \label{lemma-initial}
Given a nominal hybrid automaton $\mathcal{H}_n$ and a hybrid observer $\mathcal{O}$ with bounded estimation error in steady state, i.e., $\forall t>t_{ss}, \|\mathbf{x_e}(t)\|\leq \theta$, the volume of the initial set is also bounded, i.e., $Vol(X_I(t))\leq \Pi_{i=1}^n(2\theta+4v_i)$.
\end{proposition}
\begin{proof}
In steady state, $\forall t>t_{ss}$,
\begin{equation*}
\begin{aligned}
Vol(X_I(t)) =&\Pi_{i=1}^n(2(|\mathbf{r}^{(i)}(t)|+v_i))\leq \Pi_{i=1}^n(2(\|\mathbf{x_e}(t)\|+2v_i)) \\
\leq& \Pi_{i=1}^n(2\|\mathbf{x_e}(t)\|+4v_i) \leq \Pi_{i=1}^n(2\theta+4v_i)
\end{aligned}
\end{equation*}
\end{proof}

As discussed before, Type-$C_u$ anomaly affects the continuous outputs of the system, but can remain undetectable by residual-based methods and unobserved by discrete state observer. In order to detect this type of anomaly, we leverage the estimated states from both continuous and discrete observers, and take advantage of observation of a discrete event. This enables us to employ the contradictions among estimated continuous and discrete states and the model parameters such as guards and invariants to detect these challenging anomalies. These contradictions are formalized in Conflicts $B$ and $C$.
In what follows, we set the stage to present the main contribution of this paper, which is Theorem \ref{lemma-type2}. This theorem provides sufficient conditions on the lower bound of the anomalous signal, under which the conflict-driven method is guaranteed to detect Type-$C_u$ anomalies. Towards this goal, we first find the lower bound of the estimation error that creates one of Conflicts $B$ or $C$, and then relate this bound to the lower bound on the anomalous signal according to \eqref{eq-residualanomalous}.

Let us assume that a Type-$C_u$ anomaly occurs at time $t_f$ which causes a large estimation error on the $i_G^{th}$ state variable, i.e., $|\mathbf{x}_e^{(i_G)}| > \theta$, and a discrete event $\psi$ occurs at time $t_e$ which associates a guard with condition on the $i_G^{th}$ state variable, i.e., $\{\mathbf{x}\in Inv_q: s_G\mathbf{x}^{(i_G)}\geq c_G\}$.
Without loss of generality, we assume that the projection of $\mathcal{R}_{no,q}^{o}$ onto $\mathbf{x}_e^{(i_{G})}$ is bounded above by $c_{G}$, i.e., $\HH_{i_G}\mathcal{R}_{no,q}^{o}\leq c_{G}$ (because $s_G=1$), where $\HH_{i_G} \in \mathbb{R}^n$ is the projection row vector with the $i_G^{th}$ entry ``$1$'' and ``$0$'' elsewhere. The procedure for the case where $\HH_{i_G}\mathcal{R}_{no,q}^{o}\geq -c_{G}$ (because $s_G=-1$) is identical.  When this event occurs, we can only have two possibilities for the estimated state at time $t_e$, either $\tilde{\mathbf{x}}(t_e)\in \mathcal{R}_{no,q}^{o}$, or $\tilde{\mathbf{x}}(t_e)\in \mathcal{R}_{in}^{o}\cap Inv_q$. Based on our definitions of guard, invariant, neighbor hyperplane of the guard, and Assumption \ref{assump-nonintersecting}, along the $i_G^{th}$ state variable, the upper bound of $Inv_q$ is $c_{\mathcal{L}}$ and the lower bound of $Inv_{q'}$ is $c_G$.
For brevity in notation and as in this section we mainly consider $G(q,q',\psi)$, we refer to it as $G$.

First, consider the case where $\tilde{\mathbf{x}}(t_e)\in \mathcal{R}_{no,q}^{o}$, that is, when the real continuous state satisfies the guard, the estimated state is in the normal operating region of discrete state $q$. 
The goal is to find the lower bound of the estimation error along the $i^{th}_G$ state variable, such that:
\begin{itemize}
	\item The initial set $X_I(t_e+1)$ has no intersection with $Inv_{q'}$.
\end{itemize}
We denote such minimum estimation error corresponding to $G$ by $z^{\ast}_G$. To find $z^{\ast}_G$, it suffices to find the minimum $z$ such that for all $\tilde \x(t_e+1)$ 
the upper bound of $X_I(t_e+1)$ is smaller than the lower bound of $Inv_{q'}$ along the $i_G^{th}$ state variable,
\begin{equation} \label{eq-robust1}
\begin{aligned}
\HH_{i_G}\tilde{\mathbf{x}}(t_e+1)+\theta+2v < c_G.
\end{aligned}
\end{equation}
Note that at time $t_e$, the continuous state of the system along the $i_G^{th}$ state variable is greater than or equal to $c_G$, and smaller than the maximum value of the one time step forward reachable set from $\mathcal{P}(q,q',\psi)\cap Inv_q$ along the $i_G^{th}$ state variable, i.e., $c_G\leq \HH_{i_G}\mathbf{x}(t_e)<\epsilon$, where $\epsilon =\max(\mathbf{H}_{i_G}R_1(\mathcal{P}(q,q',\psi)\cap Inv_q))$. After the occurrence of event $\psi$, the state equation of the anomalous discrete state is changed to $(\mathbf{A}_{q'}, \mathbf{B}_{q'})$ and the estimated discrete state is changed to $q'$ at time $t_e+1$. Then the set of all possible continuous states at time $t_e+1$ can be represented by: 
\begin{equation}\label{eq-robust2}
\begin{aligned}
\forall \mathbf{x}(t_e)\in Inv_q, \,&c_G\leq \HH_{i_G}\mathbf{x}(t_e)<\epsilon, \\
&\mathbf{x}(t_e+1) \in R_1(\mathbf{x}(t_e))\subseteq \mathbf{A}_{q'}\mathbf{x}(t_e)+\Box \sigma_{q'},
\end{aligned}
\end{equation}
where $\sigma_{q'} = \|\mathbf{B}_{q'}\|\mu+w$.

Now, we can pose the problem of finding $z^{\ast}_G$ as a robust optimization problem.
\begin{equation} \label{eq-robustInit}
\begin{aligned}
z^*_G=
& \underset{z}{\min}
& & z \\
& \text{s. t.}
& & z\geq 0, \; z\geq  \HH_{i_G}  \mathbf{A}_{q'}\mathbf{x}+\sigma_{q'}+\theta+2v-c_G\\
& & & \forall \mathbf{x}\in Inv_{q'}, c_G\leq \HH_{i_G}\x \leq \epsilon, 
\end{aligned}
\end{equation}

By utilizing methods from robust optimization literature, e.g., \cite{bertsimas2016duality}, we can convert \eqref{eq-robustInit} to a linear programming problem as follows:
\begin{equation}\label{eq-dual1}
\begin{aligned}
z_G^*=
& \underset{\mathbf{J},z}{\min}
& & z \\
& \text{s. t.}
& & \left[\begin{matrix}
1\\
1
\end{matrix}\right]z-\left[\begin{matrix}
\mathbf{J}^\intercal \pmb{\rho_1}\\
0
\end{matrix}\right]\geq \left[\begin{matrix}
\sigma_{q'}+\theta+2v - \, c_G\\
0
\end{matrix}\right]\\
& & & \pmb{\Lambda}^\intercal \mathbf{J}\geq \, (\HH_{i_G}\mathbf{A}_{q'})^\intercal, \, \mathbf{J}\geq \mathbf{0}
\end{aligned}
\end{equation}
where $\mathbf{0}\in \mathbb{R}^{2n\times 1}$ is a zero vector. $\mathbf{x}$ is in a polytopic uncertain set, i.e., $\pmb{\Lambda}\mathbf{x}\leq \pmb{\rho_1}$ for problem (\ref{eq-robustInit}), where $\pmb{\Lambda} \in \mathbb{R}^{2n\times n}$, $\pmb{\rho_1}\in \mathbb{R}^{2n\times 1}$ and $\mathbf{J}\in \mathbb{R}^{2n\times 1}$ is a variable of the optimization problem.

For the second possibility, i.e., $\tilde{\mathbf{x}}(t_e)\in \mathcal{R}_{in}^{o}\cap Inv_q$, we are seeking the lower bound of the estimation error along the $i_G^{th}$ state variable such that it satisfies the following: 
\begin{itemize}
\item The reachable set for $\delta_q$ time steps from any point within the initial set $X_I(t_e)$ of the estimated continuous state has no intersection with $Inv_q$. 
\end{itemize}
Considering the worst case that the continuous state is the furthest to the upper bound of $\partial Inv_q$ along the $i_G^{th}$ state variable, i.e., $\HH_{i_G}\mathbf{x}(t_e)=c_G$, our objective can be equivalently changed to find the minimum distance between $c_G$ and $\HH_{i_G}\tilde{\mathbf{x}}(t_e)$. 
We denote this minimum distance by $d^{\ast}_G$. Define $d=|\HH_{i_G}\tilde{\x}(t_e)-c_G|$ as the distance between $\mathcal{P}(q,q',\psi)$ and the estimated state along the $i_G^{th}$ state variable. With this definition, the initial set at time $t_e$ can be represented as $X_I(t_e) = \{ \mathbf{x}: \HH_{i_G}\mathbf{x} \in [c_G+d-\theta-2v, c_G+d+\theta+2v]\}$. Starting from this initial set $X_I(t_e)$, the projection of the reachable set for $\delta_q$ time steps forward onto the $i_G^{th}$ state variable becomes $\HH_{i_G}\mathbf{A}_q^{\delta_q}\x \pm \sigma_q, \; \forall \x\in X_I(t_e)$, where $\sigma_q=\frac{1-\|\mathbf{A}_q\|^{\delta_q}}{1-\|\mathbf{A}_q\|}(\|\mathbf{B}_q\|\mu+w)$. If $\HH_{i_G}\mathbf{A}_q^{\delta_q}\x- \sigma_q > c_\mathcal{L}, \; \forall \x\in X_I(t_e)$, then it is guaranteed that the $\delta_q$ time-step forward reachable set starting from this initial set $X_I(t_e)$ has no intersection with the invariant $Inv_q$. We can pose the problem of finding $d^{\ast}_G$ as the following robust optimization problem.
\begin{equation}\label{eq-robustOpt} 
\begin{aligned}
d_G^*=
& \underset{d}{\min}
& & d \\
& \text{s. t.}
& & d\geq 0, \; \HH_{i_G} \mathbf{A}_q^{\delta_q}\mathbf{x} -\sigma_q\geq c_{\mathcal{L}}\\
& & & \forall \mathbf{x}\in Inv_q, \mathbf{x}\in X_I(t_e).
\end{aligned}
\end{equation}
With a change of variables and by employing the robust optimization techniques \cite{bertsimas2016duality}, we can write an equivalent problem to \eqref{eq-robustOpt} as a linear program.
\begin{equation}\label{eq-dual2}
\begin{aligned}
d_G^*=
& \underset{\mathbf{D},\mathbf{J}}{\min}
& & \HH_{i_G}\mathbf{D} \\
& \text{s. t.}
& & \left[\begin{matrix}
\HH_{i_G}\mathbf{A}_q^{\delta_q}\\
\HH_{i_G}
\end{matrix}\right]\mathbf{D}-\left[\begin{matrix}
\mathbf{J}^\intercal \pmb{\rho_2}\\
0
\end{matrix}\right]\geq \left[\begin{matrix}
\sigma_q+ c_{\mathcal{L}}\\
0
\end{matrix}\right]\\
& & & \pmb{\Lambda}^\intercal \mathbf{J}\geq -\, (\HH_{i_G}\mathbf{A}_q^{\delta_q})^\intercal, \, \mathbf{J}\geq \mathbf{0}, \, \mathbf{D}\geq \mathbf{0} 
\end{aligned}
\end{equation}
where $\mathbf{0}$ is a zero vector with proper dimension, and $\mathbf{D}\in \mathbb{R}^{n}$ is a vector with the $i_G^{th}$ entry $d$ and other entries ``$0$''. $\mathbf{x}$ is in a polytopic uncertain set, i.e., $\pmb{\Lambda}\mathbf{x}\leq \pmb{\rho_2}$, where $\pmb{\rho_2}\in \mathbb{R}^{2n\times 1}$ and $\mathbf{J}\in \mathbb{R}^{2n\times 1}$ is the dual variable. 

Now that we have introduced $z^{\ast}_G$ and $d^{\ast}_G$, we can present the main result of the paper.
\begin{theorem} \label{lemma-type2}
	Given a nominal hybrid automaton $\mathcal{H}_n$. Assume a Type-$C_u$ anomaly $f$ occurs at time $t_f$. If an event $\psi \in \Psi_o$ occurs at $t_e>t_f$, which is supposed to transit the system from discrete state $q$ to $q'$, and the guard $G(q,q',\psi)$ is a condition on the real continuous state which is affected by the anomaly $f$, i.e., $G(q,q',\psi): s_G\mathbf{x}^{(i_G)}\geq c_G$ and $|\mathbf{x_e}^{(i_G)}|\geq \theta$, then the conflict-driven method is guaranteed to detect the anomaly, if the anomaly satisfies:
	\begin{equation} \label{eq-anomalysig}
	\begin{aligned}
	\|\pmb{\Gamma}\pmb{\gamma}(t)\| >& \max(z^*_{q}+\theta+2v, d^*_q+\theta+2v),
	\end{aligned}
	\end{equation}
	where $z^{*}_{q} = \max_{q'} z^{*}_G$ and $d^{*}_{q} = \max_{q'} d^{*}_G$ can be derived by solving the robust optimization problems (\ref{eq-robustInit}) and (\ref{eq-robustOpt}), respectively for all possible $q'$. 
\end{theorem}

\begin{proof}
The solution $z^{*}_G$ is the lower bound of the estimation error which ensures $X_I(t_e+1)\cap Inv_{q'}=\emptyset$, i.e. Conflict $B$. The values of $z^{*}_{G}$ varies from one guard to another. Therefore, by considering $z^{*}_{q}$, we guarantee that at the discrete state $q$, regardless of guard, Conflict $B$ occurs, if $\|\pmb{\Gamma}\pmb{\gamma}(t)\| > z^*_{q}+\theta+2v$. On the other hand, the solution $d^{*}_G$ is the lower bound of the estimation error, which ensures $R_{\delta_q}(X_I(t_e))\cap Inv_{q}=\emptyset$, i.e., Conflict $C$. The values of $d^{\ast}_G$ varies for different guards, hence, we similarly take the maximum of these values for all possible $q'$, which is $d^{\ast}_q$. Based on the relation between the estimation error and anomalous signal in \eqref{eq-residualanomalous}, 
it is guaranteed that if $\|\pmb{\Gamma}\pmb{\gamma}(t)\| > d^*_{q}+\theta+2v$, regardless of guard, Conflict $C$ occurs. By combining the two conditions obtained on the magnitude of anomalous signal for the two possibilities, we can conclude that the proposed conflict-driven method provides detection guarantees on the detection of anomalous signals that satisfy condition \eqref{eq-anomalysig}, regardless of where the estimated state is located in the $Inv_q$ at the time of event. This concludes the proof. 	
\end{proof}

\section{Simulation Result} \label{sec-result}
In this section, we revisit the TG system. We present the nominal hybrid model of the TG system and compare the conflict-driven method with a residual-based method under a Type-$C_u$ anomaly. 

The graphic representation of the nominal hybrid model $\mathcal{H}_n$ of the TG system is shown in Fig.\ref{fig-tgcall}. The train automaton has one discrete state ``run''. The gate automaton has two discrete states: ``up'' and ``down'' (The time of raising and lowering the gate is ignored for simplicity). Although the automata product results in two discrete states, we additionally partition discrete state ``run, up'' to two discrete states to ensure hyperrectangle invariants as defined in \eqref{eq-invariant}. 
The discrete transitions between discrete states are determined by discrete input events $c_{up}$ and $c_{down}$, where $c_{up}$ means ``raise the gate" and $c_{down}$ means ``lower the gate". When sensor 1 detects the train and emits discrete output event $s_1$, the gate controller sends out $c_{down}$. When sensor 2 detects the train and emits discrete output event $s_2$, the gate controller sends out $c_{up}$. For each transition, we associate a guard. The invariants of the discrete states and the guards are given in Fig.\ref{fig-tgcall}. The continuous state of the TG system is $\mathbf{x} = [x_p\quad x_v]^\intercal$, where $x_p, x_v$ are the train position and the train speed, respectively. The continuous output of the TG system is $\mathbf{y}(t)=\mathbf{x}(t)+\mathbf{v}(t)$. If the train is within $16m$ of the gate, the reference speed is $0.2m/s$. Otherwise, it is $1m/s$. The desired operation is that the train speed is no faster than $0.4 m/s$ when the train is within $12 m$ of the gate. The TG system is current state observable. Based on Assumption \ref{assump-anomaytime}, we will only focus on the observer's steady state \footnote{More parameters: track length: $80m$; gate, sensor 1, sensor 2 locations: $60 m$, $45 m$, $75 m$; sampling period: $0.1s$; Upper bounds of noise: $v = 0.1$, $w = 0.01$ (Units depend on the state variable with larger noise); estimation error upper bound in the observer's steady state in nominal discrete states: $\theta=0.05$ (The unit depends on the state variable with larger estimation error at sample times).}\footnote{In reality, the train track intersects with multiple roads at different locations. The discrete state observer gives a unique estimated discrete state after passing the first road. We only focus on the track segment when the observer is in steady state.}. 
\begin{figure}[!tbp]
\vspace{-0.5em}
\centering
 \includegraphics[width=3.2in]{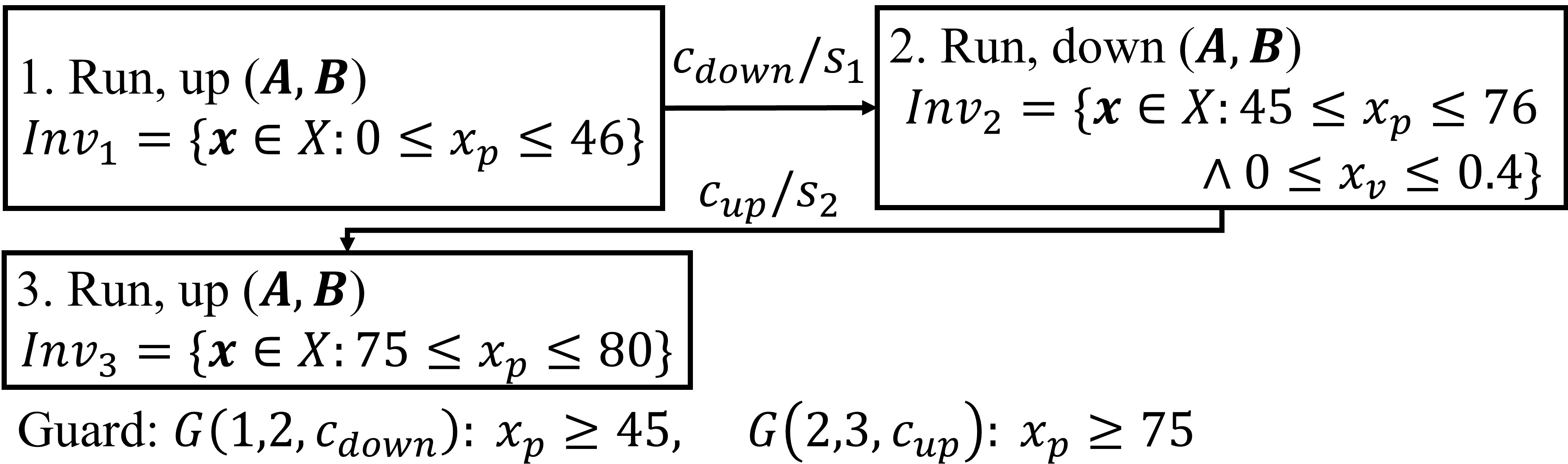}
\vspace{-0.4em}
  \caption{Hybrid automaton $\mathcal{H}_n$ of the TG system.}
  \vspace{-1em}
  \label{fig-tgcall}
\end{figure}

The intersections of the invariants give the intermediate region $\mathcal{R}_{in}$ as $\mathcal{R}_{in}= \{\forall \mathbf{x}\in X: (45\leq  x_p\leq 46 \lor 75 \leq x_p\leq 76)\land 0\leq x_v\leq 0.4 \}$. 
Then we can determine the normal operating regions of the three discrete states, as shown in Fig.\ref{fig-TGCregions},
\begin{equation}\label{eq-TGno}
\begin{aligned}
\mathcal{R}_{no,1}=  \{\forall \mathbf{x}\in X: & 0\leq x_p< 45\lor (45\leq x_p\leq 46 \land x_v>0.4)\},\\
\mathcal{R}_{no,2}=  \{\forall \mathbf{x}\in X: & 46< x_p< 75\land 0\leq x_v\leq 0.4\},\\
\mathcal{R}_{no,3}=  \{\forall \mathbf{x}\in X: & 76< x_p\leq 80\lor (75\leq x_p\leq 76 \land x_v>0.4) \}.
\end{aligned}
\end{equation}
\begin{figure}[h]
      \vspace{-0.5em}
      \centering
      \includegraphics[width=2.8in]{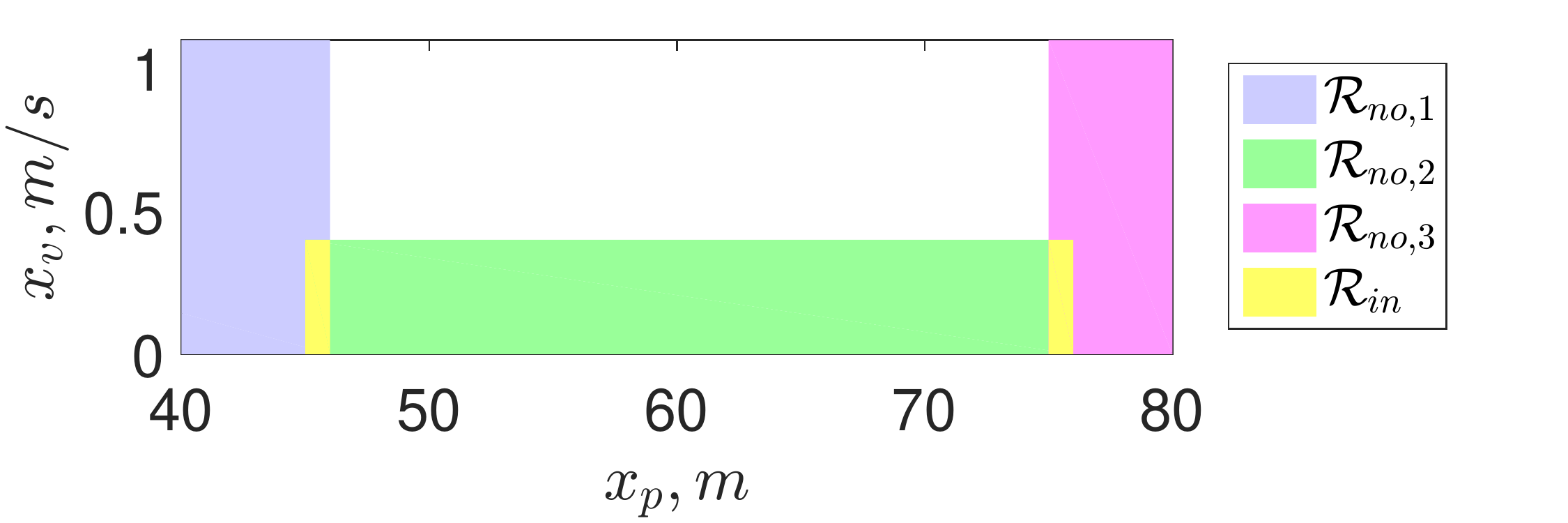}
      \vspace{-0.5em}
      \caption{The normal operating regions and the intermediate region of the TG system.} 
      \label{fig-TGCregions}
      \vspace{-1em}
\end{figure}

The neighbor hyperplane of each guard is then:
\begin{equation}
\begin{aligned}
\mathcal{L}(1,2,c_{down}): \;  x_p = 46,\quad \mathcal{L}(2,1,c_{up}): \; x_p = 76,
\end{aligned}
\end{equation}
With the invariants, guards and neighbor hyperplanes, we can determine the time step for reachability analysis of each discrete state, which is $\delta_1 = \delta_2=9, \delta_3=0$.

In these three discrete states, state matrices $(\mathbf{A}, \mathbf{B})$ are the same. The eigenvalues of $\mathbf{A}$ are 1 and 0.95. The eigenvector $\pmb{\xi}$ corresponding to the marginally stable eigenvalue is $[1 \quad 0]^\intercal$. The non-zero element of $\pmb{\xi}$ corresponds to the measured train position. 
A Type-$C_u$ anomaly scenario is a ramp anomalous signal with slope $0.02m/s$ added to the measured train position. The anomaly starts at $0s$ and runs until the end of the simulation, which makes the system violate its desired operation at $180.8 s$ with position $71.55 m$ and speed $0.41 m/s$.

The comparison of the detection performance of the conflict-driven method and a residual-based method under the anomaly mentioned above is shown in Fig.\ref{fig-comp_Cu}. The threshold of the residual-based method is $\theta+v = 0.15$ (The unit depends on $\theta$). The residual-based method fails to detect the anomaly because the residual does not increase. The conflict-driven method detects this anomaly at time $48.2 s$ when Conflict $C$ occurs. The estimated discrete state is 1, but the reachable set $R_{\delta_{1}}(X_I(482)) \cap Inv_1=\emptyset$. At $48.2s$, the norm of the anomalous signal is $0.96 m$, which is lower than the lower bound $0.98 m$ calculated by solving robust optimization problems. That means the conflict-driven method may detect the anomalies with norm lower than the lower bound which we can provide detection guarantees. 

\begin{figure}[h]
\centering
 \includegraphics[width=2.8in]{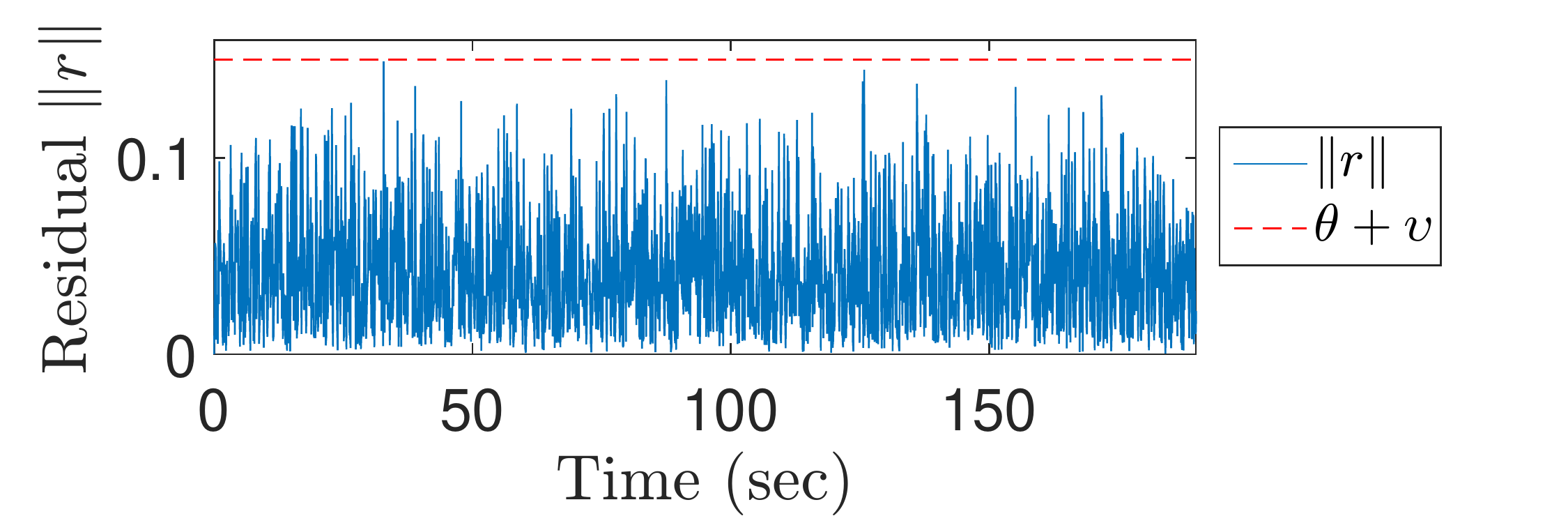}
 \includegraphics[width=3in]{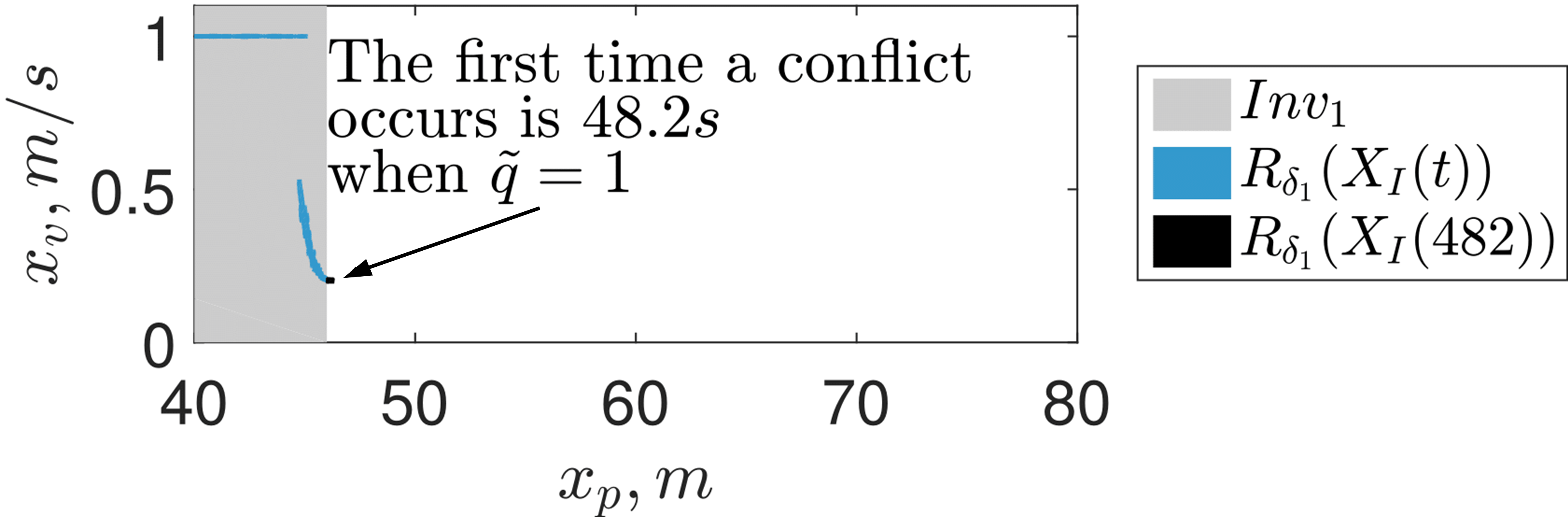}
  \caption{Simulation result under the Type-$C_u$ anomaly: (Top) Residual; (Bottom) The occurrence of Conflict $C$.}
  \label{fig-comp_Cu}
\end{figure}

\section{Conclusion and Future Work} \label{sec-conclusion}
In this paper, we propose a conflict-driven method, which uses the discrete and continuous variables and the hybrid model of the system, to provide detection guarantees for anomalies that are undetectable with traditional residual-based methods in addition to anomalies that can be detected with these methods.  
We define three different conflict types. 
If any one of the conflicts occurs, the anomaly is detected. Both mathematical demonstration and simulation result illustrate the effectiveness of the conflict-driven method. 

More work needs to be done about the conflict-driven method. 
One future work is to improve the hybrid observer design 
such that we can apply the conflict driven method to more general hybrid systems with reset maps. One potential solution is to use the Convergence Ratio method in \cite{wang2017improved}, which calculates the estimation error of the continuous state with two continuous state observers. Other future work includes the analysis of the conflict-driven method in detecting anomalies that affect both the continuous and discrete variables of the system.

\section*{Acknowledgment}
This work was supported by the National Institute
of Standards and Technology under Award No.70NANB16H205. We thank Isaac Spiegel, a Ph.D. student at University of Michigan for his discussion on hybrid system definition.

\bibliography{IEEEabrv.bib,ref.bib}{}
\bibliographystyle{IEEEtran}

\end{document}